\documentclass{amsart}

\usepackage{amsmath,amsthm,amsfonts,amssymb,amscd,latexsym}
\usepackage[all]{xy}

\newtheorem{thm}{Theorem}[section]

\newtheorem{pro}[thm]{Proposition}
\newtheorem{cor}[thm]{Corollary}

\DeclareMathOperator{\Hom}{Hom}
\DeclareMathOperator{\End}{End}
\DeclareMathOperator{\ext}{Ext}
\DeclareMathOperator{\ann}{Ann}
\DeclareMathOperator{\jac}{Jac}
\DeclareMathOperator{\soc}{Soc}
\DeclareMathOperator{\der}{Der}
\DeclareMathOperator{\gr}{gr}

\newcommand{\D}{\mathbb{D}}
\newcommand{\F}{\mathbb{F}}

\newcommand{\tS}{\overline{S}}
\newcommand{\caO}{\mathcal{O}}
\newcommand{\caF}{\mathcal{F}}

\newcommand{\argu}{\hbox to 7truept{\hrulefill}}

\begin{document}

%%%%%%%%%%%%%%%%%%%%%%%%%%%%%%%%%%%%%%%%%%%%%%%%%%%%%%%%%%%%%%%%%%%%%%%%%%%%%%%%%%%%%%%%%%%%%%%%%%%

\title[Split abelian chief factors and first degree cohomology]{Split abelian chief
factors and first degree cohomology for Lie algebras}

\author{J\"org Feldvoss}
\address{Department of Mathematics and Statistics, University of South Alabama,
Mobile, AL 36688--0002, USA}
\email{jfeldvoss@southalabama.edu}

\author{Salvatore Siciliano}
\address{Dipartimento di Matematica e Fisica ``Ennio de Giorgi", Universit\`a del Salento,
Via Provinciale Lecce-Arnesano, I-73100 Lecce, Italy}
\email{salvatore.siciliano@unisalento.it}

\author{Thomas Weigel}
\address{Dipartimento di Matematica e Applicazioni, Universit\`a degli Studi di
Milano-Bicocca, Via Roberto Cozzi, No.\ 53, I-20125 Milano, Italy}
\email{thomas.weigel@unimib.it}

\dedicatory{Dedicated to Helmut Strade on the occasion of his seventieth birthday}



%%%%%%%%%%%%%%%%%%%%%%%%%%%%%%%%%%%%%%%%%%%%%%%%%%%%%%%%%%%%%%%%%%%%%%%%%%%%%%%%%%%%%%%%%%%%%%%%%%%

\begin{abstract}
In this paper we investigate the relation between the multiplicities of split abelian
chief factors of finite-dimensional Lie algebras and first degree cohomology. In
particular, we obtain a characterization of modular solvable Lie algebras in
terms of the vanishing of first degree cohomology or in terms of the multiplicities
of split abelian chief factors. The analogues of these results are well known in
the modular representation theory of finite groups. An important tool in the proof
of these results is a refinement of a non-vanishing theorem of Seligman for the
first degree cohomology of non-solvable finite-dimensional Lie algebras in prime
characteristic. As applications we derive several results in the representation
theory of restricted Lie algebras related to the principal block and the projective
cover of the trivial irreducible module of a finite-dimensional restricted Lie algebra.
In particular, we obtain a characterization of solvable restricted Lie algebras
in terms of the second Loewy layer of the projective cover of the trivial irreducible
module.
\end{abstract}

%%%%%%%%%%%%%%%%%%%%%%%%%%%%%%%%%%%%%%%%%%%%%%%%%%%%%%%%%%%%%%%%%%%%%%%%%%%%%%%%%%%%%%%%%%%%%%%%%%%
      
\date{January 21, 2013}
          
\maketitle

%%%%%%%%%%%%%%%%%%%%%%%%%%%%%%%%%%%%%%%%%%%%%%%%%%%%%%%%%%%%%%%%%%%%%%%%%%%%%%%%%%%%%%%%%%%%%%%%%%%

\section{Introduction} 

%%%%%%%%%%%%%%%%%%%%%%%%%%%%%%%%%%%%%%%%%%%%%%%%%%%%%%%%%%%%%%%%%%%%%%%%%%%%%%%%%%%%%%%%%%%%%%%%%%%

W.\ Gasch\"utz proved the ``only if"-part of the following cohomological vanishing
theorem for finite $p$-solvable groups (see \cite[Lemma 1]{S1}), and the converse
is due to U.\ Stammbach \cite[Theorem A]{S1}. Here and in the following $p$ is an
arbitrary prime number, and $G$ is a finite group whose order is divisible by $p$.
Moreover, let $\F_p[G]$ denote the group algebra of $G$ over the field $\F_p$ with
$p$ elements, and let $C_G (M):=\{g\in G\mid g\cdot m=m\mbox{ for every }m\in M\}$
denote the {\em centralizer\/} of a (unital left) $\F_p[G]$-module $M$ in $G$.

\begin{thm}\label{gaschuetz}
A finite group $G$ is $p$-solvable if, and only if, $H^1(G/C_G(S),S)=0$ for every
irreducible $\F_p[G]$-module $S$.
\end{thm}

Let $S$ be an irreducible $\F_p[G]$-module. Then $[G:S]_{\rm split}$ denotes the
number of {\em $p$-elementary abelian chief factors\/} $G_j/G_{j-1}$ ($1\le j\le
n$) of a given chief series $\{1\}=G_0\subset G_1\subset\cdots\subset G_n=G$ that
are isomorphic to $S$ as $\F_p[G]$-modules and for which the exact sequence $\{1\}
\to G_j/G_{j-1}\to G/G_{j-1}\to G/G_j\to\{1\}$ splits in the category of groups.
According to the main result of \cite{S2}, $[G:S]_{\rm split}$ is independent of
the choice of the chief series of $G$. 

W.\ Gasch\"utz also proved the ``only if"-part of the following result on split (or
complementable) $p$-chief factors of finite $p$-solvable groups (see \cite[Theorem
VII.15.5]{HB}). The converse of Gasch\"utz' theorem is due to U.\ Stammbach
\cite[Corollary 1]{S2}), but in an equivalent form it was already proved earlier
by W.\ Willems \cite[Theorem 3.9]{W}.

\begin{thm}\label{stammbach}
A finite group $G$ is $p$-solvable if, and only if,  $\dim_{\F_p}H^1(G,S)=\dim_{\F_p}
\End_{\F_p[G]}(S)\cdot [G:S]_{\rm split}$ holds for every irreducible $\F_p[G]$-module
$S$.
\end{thm}

The main goal of this paper is to investigate whether analogues of Theorem
\ref{gaschuetz} and Theorem \ref{stammbach} hold in the context of ordinary
Lie algebras. Some time ago in an unpublished manuscript the first author
of this paper has obtained analogues of these results for restricted Lie
algebras (see \cite[Remark after Theorem 7]{F2}). The crucial point in the
argument is an analogue of Theorem \ref{gaschuetz} for ordinary Lie algebras
(see \cite[Proposition 1]{F3}). The proof in \cite {F3} was a modification
of the proof of a similar characterization of supersolvable Lie algebras due
to D.\ W.\ Barnes (see \cite[Theorem 4]{B2}). It turns out that there is a
gap in both these proofs which will be fixed in the present paper. This is
achieved by applying Shapiro's lemma for truncated coinduced modules (see
\cite[Theorem in \S5 and Corollary 1 in \S3]{D2} or \cite[Theorem 2.1 and
Theorem 1.4]{FS}) in order to establish a refinement of a non-vanishing
theorem of G.\ Seligman for the first degree cohomology of a non-solvable
finite-dimensional Lie algebra in prime characteristic (see \cite[p.\ 102]{S}).
In this regard our approach is similar to Stammbach's proof of the ``if"-part
of Theorem~\ref{gaschuetz} (see the proof of \cite[Lemma 2]{S1}).

We begin with a Lie-theoretic analogue of the main result of \cite{S2} (see
Theorem \ref{absplit}). This result is already contained in the first author's
unpublished manuscript (see \cite[Lemma 5]{F2}), but the proof there implicitly
uses that the multiplicity of split abelian chief factors isomorphic to a given
irreducible module is independent of the chief series. The proof given here follows
the argument used in the proof of \cite[Lemma 5]{F2}, but also deals with a case
not considered in \cite{F2}. As a consequence of Theorem \ref{absplit} and
Barnes' cohomological vanishing theorem for solvable Lie algebras (see
\cite[Theorem 2]{B1}), one obtains the Lie-theoretic analogue of Gasch\"utz'
theorem on split $p$-chief factors (see Theorem \ref{solv}). In the third section
we show that for fields of characteristic zero the analogues of the conditions in
Theorem \ref{gaschuetz} and Theorem \ref{stammbach} are always satisfied.
Thus, as in the group case, only modular solvable Lie algebras can be characterized
by these properties. This will be the main goal of the fourth section (see Theorem
\ref{charsolv}). In the final section we apply the results obtained in Sections 2 and
4 to restricted Lie algebras. The equivalence (i)$\Longleftrightarrow$(viii) in Theorem
\ref{pim} is an analogue of Willems' module-theoretic characterization of $p$-solvable
groups (see \cite[Theorem 3.9]{W}) for restricted Lie algebras. As a by-product we
establish several results on the second Loewy layer of the projective cover of the
trivial irreducible module. Most of the results in Section 5 were already contained in
\cite[Section 4]{F2}, but have never been published before.

%%%%%%%%%%%%%%%%%%%%%%%%%%%%%%%%%%%%%%%%%%%%%%%%%%%%%%%%%%%%%%%%%%%%%%%%%%%%%%%%%%%%%%%%%%%%%%%%%%%

\section{Split abelian chief factors and first degree cohomology}

%%%%%%%%%%%%%%%%%%%%%%%%%%%%%%%%%%%%%%%%%%%%%%%%%%%%%%%%%%%%%%%%%%%%%%%%%%%%%%%%%%%%%%%%%%%%%%%%%%%

In analogy to group theory we define a {\em chief series\/} for a finite-dimensional
Lie algebra $L$ to be an ascending chain $0=L_0\subset L_1\subset\cdots\subset
L_n=L$ of ideals in $L$ such that $L_j/L_{j-1}$ is a minimal (non-zero) ideal of
$L/L_{j-1}$ for every integer $j$ with $1\le j\le n$. Any $L_j/L_{j-1}$ is then
called a {\em chief factor\/} of $L$, and we say that $L_j/L_{j-1}$ is an {\it
abelian chief factor\/} if it is an abelian Lie algebra. For a finite-dimensional
irreducible $L$-module $S$ and a given chief series $0=L_0\subset L_1\subset\cdots
\subset L_n=L$ of $L$ we denote by $[L:S]_{\rm split}$ the number of abelian chief
factors $L_j/L_{j-1}$ that are isomorphic to $S$ as an $L$-module and for which the
exact sequence $0\to L_j/L_{j-1}\to L/L_{j-1}\to L/L_j\to 0$ splits in the category
of Lie algebras. Observe that any composition series of the adjoint $L$-module is
also a chief series of $L$ and vice versa. (But note that split chief factors do
not necessarily split in the category of $L$-modules, e.g., this happens for the
one-dimensional ideal of the non-abelian two-dimensional Lie algebra.) However, only
abelian chief factors will be relevant for our purpose. (Note that every chief factor
of a solvable Lie algebra is abelian.) As we will show that $[L:S]_{\rm split}$ is
independent of the choice of the chief series, we will not indicate the chief series
in the notation.

Our first result is completely analogous to the main result of \cite{S2} and uses
\cite[Lemma 2]{B2}:

\begin{thm}\label{absplit}
Let $L$ be a finite-dimensional Lie algebra over a field $\F$ of arbitrary characteristic,
and let $S$ be a finite-dimensional irreducible $L$-module with centralizer algebra
$\D:=\End_L(S)$. Then
\begin{equation}\label{mult}
[L:S]_{\rm split}=\dim_\D H^1(L,S)-\dim_\D H^1(L/\ann_L(S),S)
\end{equation}
holds. In particular, $[L:S]_{\rm split}$ is independent of the choice of the chief
series of $L$.
\end{thm}

\begin{proof}
We proceed by induction on the dimension of $L$. If $L$ is one-dimensional, then
the assertion is easy to check. Thus we may assume that the dimension of $L$ is
greater than one, and that the claim holds for all Lie algebras of dimension less
than $\dim_\F L$. Let $0=L_0\subset L_1\subset\cdots\subset L_n=L$ be a chief
series of $L$. For the remainder of the proof the multiplicity $[L:S]_{\rm split}$
always refers to this fixed chief series.

If $\ann_L(S)=0$, then the right-hand side of \eqref{mult} is zero. But as abelian
chief factors have non-zero annihilators, the left-hand side also vanishes and the
assertion holds. So we may assume that $\ann_L(S)\neq 0$.

Firstly, we assume that $L_1$ is contained in $\ann_L(S)$. Then the five-term exact
sequence for Lie algebra cohomology (see \cite[Theorem 6]{HS}) specializes to
\begin{equation}\label{5term}
0\to H^1(L/L_1,S)\to H^1(L,S)\to\Hom_L(L_1/[L_1,L_1],S)\to H^2(L/L_1,S)\,.
\end{equation}
Since $S$ is also an irreducible $L/L_1$-module, one obtains by induction that
\begin{equation}\label{indhyp}
[L/L_1:S]_{\rm split}=\dim_\D H^1(L/L_1,S)-\dim_\D H^1(L/\ann_L(S),S)\,.
\end{equation}
As $L_1$ is a minimal ideal of $L$, $L_1$ is either perfect or abelian.
In the former case, the third term in \eqref{5term} vanishes, and therefore
$H^1(L/L_1,S)\cong H^1(L,S)$. Since $L_1$ is not abelian, one has $[L:S]_{\rm
split}=[L/L_1:S]_{\rm split}$. Hence \eqref{mult} holds in this case.

If $L_1$ is abelian, one has $\Hom_L(L_1/[L_1,L_1],S)=\Hom_L(L_1,S)$. If $L_1$
and $S$ are not isomorphic as $L$-modules, then $\Hom_L(L_1,S)$ vanishes, and
the assertion follows as before.

For $L_1\cong S$ one has to distinguish two cases: The abelian chief factor $L_1$
is split, or not split. In case that $L_1$ is split, one has
\begin{equation}\label{add}
\begin{aligned}
&&[L:S]_{\rm split} & =[L/L_1:S]_{\rm split}+1\\
&&& = \dim_\D H^1(L/L_1,S)-\dim_\D H^1(L/\ann_L(S),S)+1\,,
\end{aligned}
\end{equation}
and \cite[Lemma 2(a)]{B2} shows that the transgression $\Hom_L(L_1,S)\to H^2
(L/L_1,S)$ is zero. Thus the exactness of \eqref{5term} implies that the
restriction $H^1(L,S)\to\Hom_L(L_1,S)$ is surjective, and therefore
\begin{equation}\label{add2}
\begin{aligned}
&& \dim_\D H^1(L,S) & =\dim_\D H^1(L/L_1,S)+\dim_\D\Hom_L(L_1,S)\\
&&& =\dim_\D H^1(L/L_1,S)+1\,.
\end{aligned}
\end{equation}
Thus \eqref{add} and \eqref{add2} yield the assertion. Suppose that $L_1$ is
not split. In this case the transgression $\Hom_L(L_1,S)\to H^2(L/L_1,S)$ is
injective (see \cite[Lemma 2(b)]{B2}). According to \eqref{5term}, the inflation
$H^1(L/L_1,S)\to H^1(L,S)$ is bijective. Then one has $[L:S]_{\rm split}=
[L/L_1:S]_{\rm split}$. Hence the claim follows from \eqref{indhyp}.

Finally, assume that $L_1$ is not contained in $\ann_L(S)$, i.e., $L_1\cap
\ann_L(S)=0$ and $S^{L_1}=0$. Suppose that $L_j/L_{j-1}$ is abelian and
$L_j/L_{j-1}\cong S$ as $L$-modules for some integer $j$ with $1\leq j\leq n$.
Then $L_j$ -- and therefore $L_1$ -- would be contained in $\ann_L(S)$, a contradiction.
Consequently, $[L:S]_{\rm split}=0$. As $S^{L_1}=0$, one concludes from the
five-term exact sequence
\begin{equation*}
0\longrightarrow H^1(L/L_1,S^{L_1})\longrightarrow H^1(L,S)\longrightarrow H^1
(L_1,S)^L\longrightarrow H^2(L/L_1,S^{L_1})
\end{equation*}
that the vertical mappings in the commutative diagram
\begin{equation*}
\label{eq:comdia}
\xymatrix{
H^1(L/\ann_L(S),S)\ar[r]^-\alpha\ar[d]&H^1(L,S)\ar[d]\\
H^1(L_1+\ann_L(S)/\ann_L(S),S)^L\ar[r]^-\beta&H^1(L_1,S)^L
}
\end{equation*}
are isomorphisms. Hence, because $\beta$ is an isomorphism, $\alpha$ is an
isomorphism as well. This shows that in this case the right-hand side of
\eqref{mult} is also zero.

Since the right-hand side of \eqref{mult} does not depend on the choice of
the chief series, the left-hand side does not either. This completes the proof
of the theorem.
\end{proof}

In the extreme case $\ann_L(S)=L$, Theorem \ref{absplit} has the following
implication:

\begin{cor} \label{triv}
Let $L$ be a finite-dimensional Lie algebra over a field $\F$ of arbitrary
characteristic. Then the trivial irreducible $L$-module occurs with multiplicity
$\dim_\F L/[L,L]$ as a split abelian chief factor of $L$.
\end{cor}

The analogue of Gasch\"utz' theorem on split $p$-chief factors (see the ``only
if"-part of Theorem \ref{stammbach} in the introduction) for solvable Lie algebras
is another immediate consequence of Theorem \ref{absplit} in conjunction with
\cite[Theorem 2]{B1} and generalizes \cite[Theorem 1]{B2}.

\begin{thm}\label{solv}
Let $L$ be a finite-dimensional solvable Lie algebra over a field $\F$ of arbitrary
characteristic. Then $$\dim_\F H^1(L,S)=\dim_\F\End_L(S)\cdot[L:S]_{\rm split}$$
holds for every finite-dimensional irreducible $L$-module $S$.
\end{thm}

The same is also true for any finite-dimensional Lie algebra in characteristic zero
(see the remark after \cite[Lemma 5]{F2}) as we will prove in the next section.

%%%%%%%%%%%%%%%%%%%%%%%%%%%%%%%%%%%%%%%%%%%%%%%%%%%%%%%%%%%%%%%%%%%%%%%%%%%%%%%%%%%%%%%%%%%%%%%%%%%

\section{Lie algebras in characteristic zero}

%%%%%%%%%%%%%%%%%%%%%%%%%%%%%%%%%%%%%%%%%%%%%%%%%%%%%%%%%%%%%%%%%%%%%%%%%%%%%%%%%%%%%%%%%%%%%%%%%%%

We begin by proving that in characteristic zero the cohomology of finite-dimen\-sional
faithful irreducible modules always vanishes (see the remark after \cite[Proposition
1]{F3}):

\begin{thm}\label{faith}
Let $L\ne 0$ be a finite-dimensional Lie algebra over a field of characteristic
zero, and let $S$ be a finite-dimensional faithful irreducible $L$-module. Then
$H^n(L,S)=0$ for every non-negative integer $n$.
\end{thm}

\begin{proof}
Since $S$ is faithful and irreducible, we have that $S^L=0$. If $L$ is semisimple,
then the assertion is an immediate consequence of Whitehead's cohomological
vanishing theorem (see \cite[Theorem 14, p.\ 96]{J2}). Otherwise $L$ has a
non-zero abelian ideal $I$.

As $S^I$ is an $L$-submodule of the faithful irreducible $L$-module $S$, we
obtain that $S^I=0$. Then it follows from \cite[Theorem 1]{B1} that $H^n(I,S)=
0$ for every positive integer $n$. Finally, from the Hochschild-Serre spectral
sequence (see \cite[Theorem 6]{HS}) one concludes that $H^n(L,S)\cong H^n
(L/I,S^I)=0$ for every positive integer $n$.
\end{proof}

In particular, we have the following characteristic zero version of Barnes'
cohomological vanishing theorem \cite[Theorem 2]{B1}:

\begin{cor}\label{van}
Let $L$ be a finite-dimensional Lie algebra over a field of characteristic zero,
and let $S$ be a finite-dimensional irreducible $L$-module. Then $$H^n(L/\ann_L
(S),S)=0$$ for every positive integer $n$.
\end{cor}

As a consequence of Theorem \ref{absplit} and Corollary \ref{van}, we obtain the
following result on split abelian chief factors and first degree cohomology in
characteristic zero which generalizes \cite[Theorem 2]{B2}:

\begin{thm}\label{char0}
Let $L$ be a finite-dimensional Lie algebra over a field $\F$ of characteristic
zero. Then $$\dim_\F H^1(L,S)=\dim_\F\End_L(S)\cdot[L:S]_{\rm split}$$ holds
for every finite-dimensional irreducible $L$-module $S$.
\end{thm}

%%%%%%%%%%%%%%%%%%%%%%%%%%%%%%%%%%%%%%%%%%%%%%%%%%%%%%%%%%%%%%%%%%%%%%%%%%%%%%%%%%%%%%%%%%%%%%%%%%%

\section{Lie algebras in prime characteristic}

%%%%%%%%%%%%%%%%%%%%%%%%%%%%%%%%%%%%%%%%%%%%%%%%%%%%%%%%%%%%%%%%%%%%%%%%%%%%%%%%%%%%%%%%%%%%%%%%%%%

The next result will be essential in proving Theorem \ref{seligman} and thereby correcting
the gap in the proof of \cite[Proposition 1]{F3}. Let $U(L)$ denote the universal enveloping
algebra of a Lie algebra $L$ over a field of prime characteristic $p$ and recall that $\caO(L,I)$
is the unital associative subalgebra of $U(L)$ generated by $I\cup\{z_1,\dots,z_k\}$, where
$ I$ is an ideal of $L$ and the elements $z_i:=e_i^{p^{m_i}}+v_i$ belong to the center of
$U(L)$, with $\{e_1,\dots,e_k\}$ a cobasis of $I$ in $L$ and $v_i\in U(L)_{(p^{m_i}-1)}$
(see \cite[p.\ 154]{FS}).

\begin{pro}\label{clifford}
Let $L$ be a finite-dimensional Lie algebra over a field $\F$ of prime characteristic, let
$I$ be an ideal in $L$ of codimension one, and let $S$ be an irreducible $I$-module.
Then every composition factor of the restriction of the truncated co-induced module
$\Hom_{\caO(L,I)}(U(L),S)$ to $I$ is isomorphic to $S$.
\end{pro}

\begin{proof} Since $I$ has codimension one in $L$, there exists $t\in L$ such that
$L=\F t\oplus I$. Consider first the induced module $M:=U(L)\otimes_{U(I)}S^*$, where
$S^*$ denotes the linear dual of $S$ which is also an irreducible $I$-module. For
any non-negative integer $n$ put $$\caF^n(M):=\sum_{0\leq\nu\leq n}t^\nu\otimes S^*\,.$$
By virtue of the Cartan-Weyl formula (see \cite[Proposition 1.1.3(4)]{SF}), $\caF^n(M)$
is an $I$-submodule of $M$ for every $n$, and $(\caF^n(M))_{n\ge 0}$ is an exhaustive
increasing filtration of $M$. Set $\gr^n(M):=\caF^n(M)/\caF^{n-1}(M)$ for any non-negative
integer $n$. Then left multiplication by $t$ induces an isomorphism $\gr^n(M)\to\gr^{n+1}
(M)$ of $I$-modules. In particular, $\gr^n(M)\cong S^*$ for every non-negative integer
$n$. Hence $(\caF^n(M))_{n\ge 0}$ is a composition series of $M$ considered as an
$I$-module.

Let $\overline{M}:=U(L)\otimes_{\caO(L,I)}S^*$ denote the truncated induced $L$-module
of $S^*$. It follows from \cite[p.\ 35]{D3} (see also \cite[Theorem 4.4(1)]{F5}) that
$\overline{M}$ is a factor module of $M$, i.e., there exists an $L$-module epimorphism
$\pi:M\to\overline{M}$. Since $\overline{M}$ is finite-dimensional, there exists a
positive integer $n$ such that $\pi$ is mapping $\caF^n(M)$ onto $\overline{M}$. Thus
every composition factor of the restriction of $\overline{M}$ to $I$ is isomorphic to
$S^*$. As $\Hom_{\caO(L,I)}(U(L),S)$ is isomorphic to $\overline{M}^*$ (see the proof
of \cite[Theorem 4.1]{F5}), every composition factor of the restriction of $\Hom_{\caO
(L,I)}(U(L),S)$ to $I$ is isomorphic to $S$.
\end{proof}

The following is a refinement of a cohomological non-vanishing theorem of G.~Se\-ligman
\cite[p.\ 102]{S} for Lie algebras that are not solvable. Note that Seligman's result is
a consequence of a result of N.~Jacobson \cite[Theorem 2]{J1} (see also \cite[Theorem 2,
p.\ 205]{J2}) stating that every finite-dimensional Lie algebra over a field of prime
characteristic has finite-dimensional modules that are not completely reducible.

\begin{thm}\label{seligman}
Let $L$ be a finite-dimensional Lie algebra over a field $\F$ of prime characteristic. If
$L$ is not solvable, then there exists an irreducible $L$-module $S$ such that $H^1(L/
\ann_L(S),S)\not=0$.
\end{thm}

\begin{proof}
Suppose that the assertion is false. Then there exists a non-solvable Lie algebra
$L$ of minimal dimension with the property that $H^1(L/\ann_L(S),S)=0$ for every
irreducible $L$-module $S$.

We first show that every proper factor algebra of $L$ is solvable. Suppose on the
contrary that $\overline{L}:=L/K$ is not solvable for some non-zero ideal $K$ of $L$.
Since $\overline{L}$ has smaller dimension than $L$, there exists an irreducible
$\overline{L}$-module $\tS$ satisfying $H^1(\overline{L}/\ann_{\overline{L}}(\tS),\tS)
\not=0$. Let $S$ denote the $L$-module which is equal to $\tS$ as a vector space
and whose $L$-action is induced by the $\overline{L}$-action on $\tS$. Hence, by
construction, $\ann_{\overline{L}}(\tS)=\ann_L(S)/K$, and therefore $H^1(L/\ann_L
(S),S)\cong H^1(\overline{L}/\ann_{\overline{L}}(\tS),\tS)\not=0$, a contradiction.

Next, we prove that $L$ has a unique minimal ideal. Suppose on the contrary that
there exist two different minimal ideals $J_1$ and $J_2$ of $L$. The minimality of
$J_1$ and $J_2$ implies that $J_1\cap J_2=0$, and thus $L$ can be embedded into
$L/J_1\times L/J_2$ which is solvable. This is a contradiction, and therefore $L$
has a unique minimal ideal which we denote by $\soc(L)$.

Since $\soc(L)$ is non-zero, $L/\soc(L)$ is solvable. Moreover, $\soc(L)$ is not abelian
and $L/\soc(L)\neq 0$. Suppose on the contrary that $L=\soc(L)$. As $\soc(L)$ is the
unique minimal ideal of $L$, it is contained in every non-zero ideal of $L$, and so we
obtain that $L$ is simple. According to a result of Seligman \cite[p.\ 102]{S} and the
long exact sequence in cohomology, there exists an irreducible $L$-module $S$
satisfying $H^1(L,S)\not=0$. As $L$ is perfect, one has $H^1(L,\F)=0$. Hence $S$
is non-trivial, and by the simplicity of $L$, we obtain that $\ann_L(S)=0$. Consequently,
$L$ cannot be a counterexample to the assertion of the theorem which again is a contradiction.

Now we prove that $L$ has an ideal $I$ of codimension one. In particular, $I$ is not
solvable. Since $L/\soc(L)$ is a non-zero solvable Lie algebra, one concludes that $L/[L,L]
\not=0$. Choose a subspace $I$ in $L$ of codimension one that contains $[L,L]$. Then
$I$ is an ideal in $L$ of codimension one.

As $I$ is not solvable of codimension one in $L$, there exists an irreducible $I$-module
$S$ satisfying $H^1(I/\ann_I(S),S)\not=0$. According to \cite[Theorem 2]{B1}, $I/\ann_I
(S)$ is not solvable. Since $I/\soc(L)$ is a solvable Lie algebra, $\soc(L)\not\subseteq
\ann_I(S)$. As a consequence, $\ann_I(S)=0$, and thus also $S^I=0$ as well as $H^1
(I,S) \not=0$.

It follows from \cite[Theorem in \S5 and Corollary 1 in \S3]{D2} or \cite[Theorem 2.1
and Theorem 1.4]{FS} that
\begin{equation*}
H^1(L,\Hom_{\caO(L,I)}(U(L),S))\cong H^1(I,S)\oplus (L/I)\otimes S^I=H^1(I,S)\ne 0\,.
\end{equation*}
Hence by the long exact sequence in cohomology there exists an irreducible composition
factor $X$ of the $L$-module $\Hom_{\caO(L,I)}(U(L),S)$ such that $H^1(L,X)\not=0$.

According to Proposition \ref{clifford}, every composition factor of the restriction of
$\Hom_{\caO(L,I)}(U(L),S)$ to $I$ is isomorphic to $S$. Since $X$ is a composition
factor of $\Hom_{\caO(L,I)}(U(L),S)$, every composition factor of the restriction of
$X$ to $I$ is also isomorphic to $S$. In particular, the socle of $X$ is isomorphic
to a direct sum of copies of $S$. Consequently, $\ann_I(X)\subseteq\ann_I(S)=0$.
Suppose that $\ann_L(X)\neq 0$. Then $0\neq\soc(L)\subseteq\ann_L(X)\cap I=
\ann_I(X)$ which is a contradiction. Hence $\ann_L(X)=0$. But this shows that $X$
is a faithful irreducible $L$-module such that $H^1(L,X)\not=0$ contradicting the
choice of $L$.
\end{proof}

We are ready to prove the following characterization of solvable Lie algebras over
fields of prime characteristic which is the Lie-theoretic analogue of \cite[Theorem
A]{S1} and \cite[Corollary 1]{S2} (see Theorem \ref{gaschuetz} and Theorem
\ref{stammbach} in the introduction). The equivalence of (i) and (ii) in Theorem
\ref{charsolv} also corrects the gap in the proof of \cite[Proposition 1]{F3}.

\begin{thm}\label{charsolv}
Let $L$ be a finite-dimensional Lie algebra over a field $\F$ of prime characteristic.
Then the following statements are equivalent:
\begin{enumerate}
\item [(i)]   $L$ is solvable.
\item [(ii)]  $H^1(L/\ann_L(S),S)=0$ for every irreducible $L$-module $S$.
\item [(iii)] $\dim_\F H^1(L,S)=\dim_\F\End_L(S)\cdot[L:S]_{\rm split}$ holds for every
                  irreducible $L$-module $S$.
\end{enumerate}
\end{thm}

\begin{proof}
The implication (i)$\Longrightarrow$(ii) is just \cite[Theorem 2]{B1}, and the equivalence
of (ii) and (iii) follows from Theorem \ref{absplit}. Finally, the remaining implication
(ii)$\Longrightarrow$(i) is just the contraposition of the statement of Theorem \ref{seligman}.
\end{proof}

\noindent {\bf Remark.} Both the proofs of \cite[Theorem 4]{B2}) and \cite[Proposition
1]{F3} use that for a finite-dimensional Lie algebra $L$ with unique minimal ideal $I$
the vanishing of $H^1(I,I)^L$ implies that $[L,L]\subseteq I$. This is not true as the
following example shows: For $L:=\der(\mathfrak{psl}_3(\F))$ and $I:=\mathfrak{psl}_3
(\F)$ over any field $\F$ of characteristic three one has that $L/I\cong I$, and $I$ is the
unique minimal ideal of $L$. Moreover, $H^1(I,I)^L=0$, but $[L,L]=L$.

The statement of \cite[Proposition 1]{F3} is the equivalence of (i) and (ii) in Theorem
\ref{charsolv}. A proof of \cite[Theorem 4]{B2} can be obtained along the lines of the
argument in \cite{F3}, when \cite[Proposition 1]{F3} is replaced by Theorem \ref{charsolv}.
\vspace{.3cm}

Corollary \ref{van} and Theorem \ref{char0} show that (ii) and (iii) in Theorem \ref{charsolv}
are always true in characteristic zero. So these conditions do not characterize solvable Lie
algebras in this case.

%%%%%%%%%%%%%%%%%%%%%%%%%%%%%%%%%%%%%%%%%%%%%%%%%%%%%%%%%%%%%%%%%%%%%%%%%%%%%%%%%%%%%%%%%%%%%%%%%%%

\section{Restricted Lie algebras}

%%%%%%%%%%%%%%%%%%%%%%%%%%%%%%%%%%%%%%%%%%%%%%%%%%%%%%%%%%%%%%%%%%%%%%%%%%%%%%%%%%%%%%%%%%%%%%%%%%%

Let $A$ be a finite-dimensional (unital) associative algebra with Jacobson radical
$\jac(A)$, and let $M$ be a (unital left) $A$-module. Then the descending filtration
$$M\supset\jac(A)M\supset\jac(A)^2M\supset\jac(A)^3M\supset\cdots\supset\jac
(A)^{\ell}M\supset\jac(A)^{\ell+1}M=0$$ is called the {\em Loewy series\/} of $M$
and the factor module $\jac(A)^{n-1}M/\jac(A)^nM$  is called the $n^{\rm th}$ {\em
Loewy layer\/} of $M$ (see \cite[Definition 1.2.1]{B} or \cite[Definition VII.10.10a)]{HB}).

Recall that a projective module $P_A(M)$ is a {\it projective cover\/} of $M$, if there
exists an $A$-module epimorphism $\pi_M$ from $P_A(M)$ onto $M$ such that
the kernel of $\pi_M$ is contained in the radical $\jac(A)P_A(M)$ of $P_A(M)$. If
projective covers exist, then they are unique up to isomorphism. It is well known that
projective covers of finite-dimensional modules over finite-dimensional associative
algebras exist and are again finite-dimensional. Moreover, every projective
indecomposable $A$-module is isomorphic to the projective cover of some irreducible
$A$-module. In this way one obtains a bijection between the isomorphism classes
of the projective indecomposable $A$-modules and the isomorphism classes of the
irreducible $A$-modules.

Let $L$ be a finite-dimensional restricted Lie algebra over a field $\F$ of prime
characteristic, and let $u(L)$ denote the restricted universal enveloping algebra
of $L$ (see \cite[p.\ 192]{J2} or \cite[p.\ 90]{SF}). Then every restricted $L$-module
is an $u(L)$-module and vice versa, and so there is a bijection between the irreducible
restricted $L$-modules and the irreducible $u(L)$-modules. In particular, as $u(L)$
is finite-dimensional (see \cite[Theorem 12, p.\ 191]{J2} or \cite[Theorem 2.5.1(2)]{SF}),
every irreducible restricted $L$-module is finite-dimensional. In the sequel we use
the notation $P_L(\F):=P_{u(L)}(\F)$ for the projective cover of the trivial irreducible
$L$-module. Following Hochschild \cite{H} we define the {\em restricted cohomology\/}
of $L$ with coefficients in a restricted $L$-module $M$ by $H_*^n(L,M):=\ext_{u(L)}^n
(\F,M)$ for every non-negative integer $n$.

Using \cite[Proposition 2.4.3]{B} and Theorem \ref{absplit} we obtain a lower bound
for the multiplicity of a non-trivial irreducible restricted $L$-module in the second
Loewy layer of $P_L(\F)$ (see \cite[Theorem 3.7]{W} for the analogue in the modular
representation theory of finite groups):

\begin{thm}\label{loewybd}
Let $L$ be a finite-dimensional restricted Lie algebra over a field $\F$ of prime
characteristic. Then $$[\jac(u(L))P_L(\F)/\jac(u(L))^2P_L(\F):S]\ge[L:S]_{\rm split}$$
holds for every non-trivial irreducible restricted $L$-module $S$.
\end{thm}

\begin{proof}
Since $S$ is not trivial, it follows from the beginning of Hochschild's six-term
exact sequence relating ordinary and restricted cohomology \cite[p.\ 575]{H} (see
also \cite[Lemma 2(b)]{F3}) that $H_*^1(L,S)\cong H^1(L,S)$. Then we obtain from
\cite[Proposition 2.4.3]{B} and Theorem \ref{absplit} that
\begin{align*}
\dim_\F\End_L(S)&\cdot[\jac(u(L))P_L(\F)/\jac(u(L))^2P_L(\F):S]\\
&=\dim_\F\ext_{u(L)}^1(\F,S)=\dim_\F H_*^1(L,S)=\dim_\F H^1(L,S)\\
&\ge\dim_\F H^1(L,S)-\dim_\F H^1(L/\ann_L(S),S)\\
&=\dim_\F\End_L(S)\cdot[L:S]_{\rm split}\,.
\end{align*}
Cancelling $\dim_\F\End_L(S)$ yields the desired inequality.
\end{proof}

\noindent {\bf Remark.} If one uses the main result of \cite{S2} instead of Theorem
\ref{absplit}, then the above proof would also work in the case of finite-dimensional
modular group algebras. This provides an alternative proof of \cite[Theorem 3.7]{W}.  
\vspace{.3cm}

\noindent {\bf Example.} Consider the three-dimensional simple Lie algebra $L:=
\mathfrak{sl}_2(\F)$ over an algebraically closed field $\F$ of characteristic $p>2$.
Take for $S$ the $(p-1)$-dimensional irreducible restricted $L$-module. Then $[\jac
(u(L))P_L(\F)/\jac(u(L))^2P_L(\F):S]=2$ (see \cite[Theorem 1(ii)]{P}), but $[L:
S]_{\rm split}=0$. This shows that equality does not necessarily hold in Theorem
\ref{loewybd}.

Consider the non-abelian two-dimensional restricted Lie algebra $$L:=\F t\oplus\F e,
\qquad [t,e]=e,~~t^{[p]}=t,~~e^{[p]}=0\,.$$ Then it is well known that $P_L(\F)\cong
u(L)\otimes_{u(\F t)}\F$ (see \cite[Satz II.3.2]{F0} and \cite[Proposition 2.2]{F}).
As the trivial irreducible $L$-module only occurs in the top of $P_L(\F)$, we have
that $[\jac(u(L))P_L(\F)/\jac(u(L))^2P_L(\F):\F]=0$. This in conjunction with Corollary
\ref{triv} shows that Theorem \ref{loewybd} is not true for the trivial irreducible
$L$-module.
\vspace{.3cm}

As an immediate consequence of Theorem \ref{loewybd}, we obtain the following weak
analogue of a well-known result for finite modular group algebras:

\begin{cor}\label{splitsolv}
Every non-trivial split chief factor of a finite-dimensional solvable restricted Lie
algebra $L$ is a direct summand of the second Loewy layer of the projective cover $P_L
(\F)$ of the trivial irreducible $L$-module. In particular, every split chief factor
of a finite-dimensional solvable restricted Lie algebra $L$ is a composition factor
of $P_L(\F)$.
\end{cor}

Recall that the {\em principal block\/} of a restricted Lie algebra is the block that
contains the trivial irreducible module.
\vspace{.3cm}

\noindent {\bf Question.} In view of Corollary \ref{splitsolv}, it is natural to ask
whether every chief factor of a finite-dimensional solvable restricted Lie algebra
$L$ is a composition factor of $P_L(\F)$, or even more generally (see Proposition
\ref{chiefpriblo} below), whether every irreducible module in the principal block
of $u(L)$ is a composition factor of $P_L(\F)$ (for an affirmative answer to the
analogous question in the modular representation theory of finite $p$-solvable groups
see \cite[Theorem VII.15.8]{HB}).
\vspace{.3cm}

Let $\langle X\rangle_{\F}$ denote the $\F$-subspace of $L$ spanned by a subset $X$
of $L$. Using in addition \cite[Theorem 2]{B1} and \cite[Proposition 2.7]{F1} we obtain
in the solvable case the following more precise result (for the analogous result of
W.\ Gasch\"utz in the modular representation theory of finite $p$-solvable groups see
\cite[Theorem VII.15.5b)]{HB}):

\begin{pro}\label{LLPIM}
Let $L$ be a finite-dimensional solvable restricted Lie algebra over a field $\F$ of
prime characteristic. If $S$ is an irreducible restricted $L$-module, then
\begin{eqnarray*}
\lefteqn{ [\jac(u(L))P_L(\F)/\jac(u(L))^2P_L(\F):S]}\\
&& =\left\{ \begin{array}{ll}
[L:S]_{\rm split} & \mbox{\rm if }S\not\cong\F\\
{}[L:S]_{\rm split}-\dim_{\F}(\langle L^{[p]}\rangle_{\F}/[L,L]\cap\langle L^{[p]}
\rangle_{\F}) & \mbox{\rm if }S\cong\F
\end{array}\right. .
\end{eqnarray*}
\end{pro}

\noindent {\bf Remark.} Note that the second equality in Proposition \ref{LLPIM} remains
true for an arbitrary finite-dimensional restricted Lie algebra. This implies that the
reverse inequality of Theorem \ref{loewybd} holds for the trivial irreducible module.
\vspace{.3cm}

For the convenience of the reader we include a proof of the following result which in
the case of solvable restricted Lie algebras is already contained in a previous paper
of the first author (see \cite[Proposition 2]{F4}).

\begin{pro}\label{chiefpriblo}
Every abelian chief factor of a finite-dimensional restricted Lie algebra $L$ belongs to
the principal block of $L$.
\end{pro}

\begin{proof}
Let $S=I/J$ be an abelian chief factor of $L$. In particular, $S$ is a trivial $I$-module.
Then the five-term exact sequence for ordinary cohomology (see \cite[Theorem 6]{HS})
yields $$0\to H^1(L/I,S)\to H^1(L/J,S)\to\Hom_L(S,S)\to H^2(L/I,S)\to H^2(L/J,S)\,.$$
Since the third term is non-zero, the second or fourth term must also be non-zero. In
the first case we obtain from Hochschild's six-term exact sequence relating ordinary
and restricted cohomology \cite[p.\ 575]{H} (see also \cite[Lemma 2(b)]{F3}) and the
five-term exact sequence for restricted cohomology that either $S\cong\F$ or $0\ne
H_*^1(L/J,S)\hookrightarrow H_*^1(L,S)$. Therefore $S$ belongs to the principal block
of $L$ (see \cite[Lemma 1(b)]{F3}). In the second case Hochschild's six-term exact
sequence relating ordinary and restricted cohomology yields that either $H^1(L,S)\ne
0$ or $H_*^2(L/I,S)\ne 0$. In the former case either $S\cong\F$ or $H_*^1(L,S)\ne 0$,
and thus $S$ belongs to the principal block of $L$. In the latter case $S$ belongs to
the principal block of $L/I$, and by \cite[Lemma 4]{F4}, $S$ also belongs to the principal
block of $L$.
\end{proof}

Finally, we obtain the following characterization of solvable restricted Lie algebras which
was motivated by \cite[Theorem 3.9]{W} (see also \cite[Remark after Theorem 7]{F2})
for the equivalence of (i), (vi), and (viii) as well as \cite[Theorem 1]{F3} for the equivalence
of (i), (iv), and (v)).

\begin{thm}\label{pim}
Let $L$ be a finite-dimensional restricted Lie algebra over a field $\F$ of prime
characteristic. Then the following statements are equivalent:
\begin{enumerate}
\item[(i)]    $L$ is solvable.
\item[(ii)]   $H^1(L/\ann_L(S),S)=0$ for every non-trivial irreducible restricted
                  $L$-module $S$.
\item[(iii)]  $\dim_\F H^1(L,S)=\dim_\F\End_L(S)\cdot[L:S]_{\rm split}$ holds for
                  every non-trivial irreducible restricted $L$-module $S$.
\item[(iv)]  $H_*^1(L/\ann_L(S),S)=0$ for every non-trivial irreducible restricted
                  $L$-module $S$.
\item[(v)]   $H_*^1(L/\ann_L(S),S)=0$ for every non-trivial irreducible restricted
                  $L$-module $S$ belonging to the principal block of $L$.
\item[(vi)]  $\dim_\F H_*^1(L,S)=\dim_\F\End_L(S)\cdot[L:S]_{\rm split}$ holds for
                  every non-trivial irreducible restricted $L$-module $S$.
\item[(vii)] $\dim_\F H_*^1(L,S)=\dim_\F\End_L(S)\cdot[L:S]_{\rm split}$ holds for
                  every non-trivial irreducible restricted $L$-module $S$ belonging to
                  the principal block of $L$.                  
\item[(viii)] $[\jac(u(L))P_L(\F)/\jac(u(L))^2P_L(\F):S]=[L:S]_{\rm split}$ holds for
                   every non-trivial irreducible restricted $L$-module $S$.
\item[(ix)]  $[\jac(u(L))P_L(\F)/\jac(u(L))^2P_L(\F):S]=[L:S]_{\rm split}$ holds for
                  every non-trivial irreducible restricted $L$-module $S$ belonging to the
                  principal block of $L$.
\end{enumerate}
\end{thm}

\begin{proof}
The equivalence of (i) and (ii) follows from Theorem \ref{charsolv} in conjunction
with \cite[Theorem 2]{D1} (see also \cite[Corollary 3.2]{F5}) and $\ann_L(\F)=L$.
The equivalence of (ii) and (iii) is an immediate consequence of Theorem \ref{absplit}.
The equivalences of (ii) and (iv) as well as (iii) and (vi) follow from Hochschild's
six-term exact sequence relating ordinary and restricted cohomology \cite[p.\ 575]{H}
(see also \cite[Lemma 2(b)]{F3}). Moreover, \cite[Lemma 4]{F4} and \cite[Lemma 1(a)]{F3}
yield the equivalence of (iv) and (v). This in conjunction with Proposition \ref{chiefpriblo}
also shows the equivalence of (vi) and (vii). Finally, the equivalences of (vi) and (viii)
as well as (vii) and (ix) are both consequences of $\dim_\F\End_L(S)\cdot[\jac(u(L))P_L
(\F)/\jac(u(L))^2P_L(\F):S]=\dim_\F H_*^1(L,S)$.
\end{proof}

It is not surprising that the results in this section do not include the trivial
irreducible module. For this it does not suffice to consider abelian chief factors
in the category of ordinary Lie algebras, but one has to consider strongly abelian
$p$-chief factors in the category of restricted Lie algebras. We will investigate
this in more detail on another occasion.
\vspace{.3cm}

\noindent {\bf Acknowledgements.} The first and the second author would like to thank
the Dipartimento di Matematica e Applicazioni at the Universit\`a degli Studi di
Milano-Bicocca for the hospitality during their visit in May 2012 when large portions
of this paper were written.

%%%%%%%%%%%%%%%%%%%%%%%%%%%%%%%%%%%%%%%%%%%%%%%%%%%%%%%%%%%%%%%%%%%%%%%%%%%%%%%%%%%%%%%%%%%%%%%%%%%  

%%%%%%%%%%%%%%%%%%%%%%%%%%%%%%%%%%%%%%%%%%%%%%%%%%%%%%%%%%%%%%%%%%%%%%%%%%%%%%%%%%%%%%%%%%%%%%%%%%%

\end{document}